\documentclass{llncs}
\usepackage{caption} 

\usepackage[numbers,sort]{natbib}
\usepackage{amsmath,amssymb,amsfonts}
\usepackage{color}
\usepackage[colorlinks=true,breaklinks=true,bookmarks=true,urlcolor=blue,
     citecolor=blue,linkcolor=blue,bookmarksopen=false,draft=false]{hyperref}
\usepackage{bbm}

\newcommand{\Comments}{1}
\definecolor{gray}{gray}{0.5}
\definecolor{darkgreen}{rgb}{0,0.5,0}
\newcommand{\mynote}[2]{\ifnum\Comments=1\textcolor{#1}{#2}\fi}

\newcommand{\E}{\mathbb{E}}

\renewcommand{\O}{\mathcal{O}}
\renewcommand{\o}{\mathit{o}}
\newcommand{\R}{\mathcal{R}}
\newcommand{\T}{\mathcal{T}}

\newcommand{\ones}{\mathbbm{1}}

\newcommand{\cell}{\mathrm{cell}}

\newcommand{\reals}{\mathbb{R}}
\newcommand{\argmin}{\mathop{\mathrm{argmin}}}
\newcommand{\argmax}{\mathop{\mathrm{argmax}}}

\usepackage{tikz,pgfplots,tikz-3dplot}
\usetikzlibrary{arrows.meta,shapes.misc,positioning,fit,calc,intersections,through,decorations.markings}
\tdplotsetmaincoords{55}{135}        
\tikzstyle{cell}=[dashed,thick]
\tikzstyle{simplex}=[thick]
\newcommand{\tikzfigscale}{2.9}
\newcommand{\placefiglabel}[1]{\node at (.9,0,1.4) {#1};}

\title{Power Diagram Detection with \\Applications to Information Elicitation}

\author{Steffen Borgwardt\inst{1}\and Rafael M. Frongillo\inst{2}}

\institute{\email{\href{mailto:steffen.borgwardt@ucdenver.edu}{steffen.borgwardt@ucdenver.edu}};
University of Colorado Denver \and
\email{\href{mailto:raf@colorado.edu}{raf@colorado.edu}};
University of Colorado Boulder }

\begin{document}

\maketitle

\begin{abstract}
Power diagrams, a type of weighted Voronoi diagrams, have many applications throughout operations research. We study the problem of power diagram detection: determining whether a given finite partition of $\reals^d$ takes the form of a power diagram. This detection problem is particularly prevalent in the field of information elicitation, where one wishes to design contracts to incentivize self-minded agents to provide honest information.

We devise a simple linear program to decide whether a polyhedral cell decomposition can be described as a power diagram. Further, we discuss applications to property elicitation, peer prediction, and mechanism design, where this question arises. Our model is able to efficiently decide the question for decompositions of $\reals^d$ or of a restricted domain in $\reals^d$. The approach is based on the use of an alternative representation of power diagrams, and invariance of a power diagram under uniform scaling of the parameters in this representation.
\end{abstract}

\section{Introduction and Background}

Power diagrams play an important role in many disciplines within operations research, ranging from balanced least-squares clustering~\cite{bbg-14, bbg-17,bbd-00} to multiclass classification~\cite{bm-92, b-14,bb-99,cs-01,v-98,ww-98}. They have recently surfaced in the domain of information elicitation, where one wishes to design contracts or mechanisms to incentivize a self-minded agent to reveal their private information truthfully~\cite{lambert2009eliciting,lambert2011elicitation,frongillo2014general,agarwal2015consistent,frongillo2017geometric}.
In the latter, power diagrams are shown in various settings to characterize the possible sets of information that can be treated identically under the mechanism, an important step in understanding which mechanisms can and cannot be \emph{truthful}, i.e.,\ elicit this information effectively.
Motivated by this application in particular, where one wishes to know whether a given mechanism or contract could be truthful, we study the problem of power diagram detection: deciding whether a given cell decomposition is in fact a power diagram. 

We begin with a formal definition of a power diagram~\cite{aurenhammer1987power}. 

\begin{definition}[Power Diagram, Original]
  \label{def:pd}
  Given a set of sites $s_1,\ldots,s_k\in\reals^d$ and weights $w_1,\ldots,w_k\in\reals$, we define the cell corresponding to site $s_i$ as
\begin{equation}
  \label{eq:cell}
  \cell(s_i) = \left\{\, x\in\reals^d : i \in \argmin\nolimits_j \|x-s_j\|^2-w_j^2\, \right\}~.
\end{equation}
A cell decomposition $P=(P_1,\ldots,P_k)$ of $\reals^d$ is a \emph{power diagram} if there exist $\{s_i\}_{i=1}^k \subseteq \reals^d$ and $\{w_i\}_{i=1}^k \subseteq \reals$ so that for all $i\in\{1,\ldots,k\}$ we have $P_i = \cell(s_i)$, as in Eq.~\eqref{eq:cell}.  
\end{definition}
Informally, the cell of a power diagram is the set of all points which are closest to the corresponding site in squared Euclidean distance, discounting each site by its weight squared. For the sake of simplicity, we assume that cell decompositions in $\mathbb{R}^d$ have full-dimensional cells.

Several basic observations follow from this definition. First, note that $\cell(s_i)$ depends on \emph{all} sites $s_1,\ldots,s_k$ and weights $w_1,\ldots,w_k$. Second, setting $w_1 = \cdots = w_k$ recovers the definition of a Voronoi diagram, where points are assigned to a partition according to which site is closest in Euclidean distance. Third, any power diagram is a polyhedral cell complex, in the sense that all of its cells, and the interesections of these cells, are polyhedra. In particular, the shared boundary of any pair of cells is either empty or is contained in a hyperplane, and thus can be described via linear constraints. Power diagrams exhibit many useful properties beyond the basic observations above, see e.g.~\cite{aurenhammer1987criterion,aurenhammer1987power,b-10,b-14,bg-12}.

The problem we consider is the following: given some polyhedral cell complex $P$, where the cells are described by linear inequalities, determine whether or not $P$ is a power diagram. In Section~\ref{sec:lp}, we devise a novel linear program to solve this problem with a running time complexity that is (at most) weakly polynomial. We treat two variants of the problem, where $P$ decomposes $\reals^d$ and where it decomposes a restricted convex domain $D\subseteq \reals^d$. We then present and discuss several applications to information elicitation in Section~\ref{sec:applications}. We conclude in Section~\ref{sec:conclusion} with a brief review of our contribution and additional challenges in practice.

\section{Power Diagram Detection by Linear Programming}\label{sec:lp}

We devise a linear program to check whether a given polyhedral cell complex can be represented as a power diagram. For this, recall that polyhedra can be described as the intersection of finitely many halfspaces. Let us begin with a formal problem statement.

\begin{problem}[Power Diagram Detection]\label{problem}
Let $P=(P_1,\dots,P_k)$ be a polyhedral cell complex defined by a set of separating hyperplanes between the cells. Let the index sets $J_i = \{j : \dim(P_i \cap P_j)=d-1\}$ give the adjacency structure in the complex. Finally, let $a_{ij}\in\reals^d$ and $\gamma_{ij}\in \reals$ with \begin{equation}
    \label{eq:polyhedral-def}
    P_i = \{x\in\reals^d : a_{ij}^T x \leq \gamma_{ij} \; \forall j \in J_i\}.
  \end{equation}
\emph{Decide:} Is there a set of sites $\{s_i\}_{i=1}^k \subset \reals^d$ and $\{w_i\}_{i=1}^k \subset \reals$ such that $P_i = \cell(s_i)$ for all $i\in\{1,\ldots,k\}$?
\end{problem}
Related questions have been studied (for stronger input) in the literature: Based on the information in an incidence lattice of a cell complex, it is possible to determine whether the cell complex is the Voronoi diagram of a set of sites in time linear in the number of facets of the cell complex~\cite{aurenhammer1987recognising}. A similar result, with time linear in the number of vertices of the cell complex, holds in the dual setting of Dirichlet tesselations~\cite{ab-85}. Rybnikov presented an algorithm to decide whether a cell complex given through an incidence graph with information on facets and $(d{-}2)$--faces is a power diagram~\cite{rybnikov1999stresses}. The algorithm is based on finding a feasible point in a set of equalities and strict inequalities.
Its running time is weakly polynomial provided one can address its numerical instability and the challenging implementation of a tailored interior point method.

In this paper, we devise a simple solution to Problem~\ref{problem} that will give a straightforward resolution for applications such as those outlined in Section~\ref{sec:applications}.
The advantages lie in a simple implementation, solvability by any linear programming algorithm, numerical stability, and a practical running time that matches its (favorable) theoretical bound; see Theorems~\ref{thm:detect} and~\ref{thm:detect-restricted}.

Our first tool is the use of an alternative definition of power diagrams, which has been devised in different ways, ranging from pairwise linear separation in a multiclass setting ~\cite{bm-92,b-10,b-14} to duality theory~\cite{bg-12}.

\begin{definition}[Power Diagram, Alternative]\label{def:pd2}
Given a set of sites $s_1,\ldots,s_k\in\reals^d$ and $\gamma_1,\dots,\gamma_k \in \reals$, we define the cell corresponding to site $s_i$ as
\begin{equation}
  \label{eq:gammacells}
  \cell(s_i) = \{\, x\in\reals^d : (s_j-s_i)^Tx\leq \gamma_j-\gamma_i \; \forall j \neq i  \}~.
\end{equation}
A cell decomposition $P=(P_1,\ldots,P_k)$ of $\reals^d$ is a \emph{power diagram} if there exist $\{s_i\}_{i=1}^k \subseteq \reals^d$ and $\gamma_1,\dots,\gamma_k \in \reals$ so that for all $i\in\{1,\ldots,k\}$ we have $P_i = \cell(s_i)$, as in Eq.~\eqref{eq:gammacells}.  
\end{definition}

It is easy to match this definition with the original notion of a power diagram through a small transformation~\cite{b-14}: The cells $P_i$ and $P_j$ are separated by the hyperplane \begin{eqnarray*}
H_{ij}& := &\{x \in \mathbb{R}^d: \|s_i-x\|^2 -w_i = \|s_j-x\|^2-w_j\} \\
& = & \{x \in \mathbb{R}^d: 2(s_j-s_i)^Tx = (s_j^Ts_j-w_j)-(s_i^Ts_i-w_i)\}
\end{eqnarray*}
Choosing $\gamma_i=\frac{1}{2}(s_i^Ts_i-w_i)$ for $i \leq k$ yields this new representation, which fits with the definition of piecewise-linear separability~\cite{bm-92} with the small exception that we use weak instead of strict inequalities.
The right-hand sides in the above take the form $\gamma_j-\gamma_i$. This means that the ability to devise a set of sites $\{s_i\}_{i=1}^k \subseteq \reals^d$ and $\gamma_1,\dots,\gamma_k$ would be a positive resolution to Problem~\ref{problem}. 

Our second tool is the invariance of power diagrams under scaling~\cite{aha-98,b-10}. More precisely, the above representation of power diagrams is invariant under scaling of all $s_i$ and $\gamma_i$, for all $i\leq k$, with the same parameter. For our purposes, it is important to note that in the representation of $H_{ij}$, all parameters can be scaled with a $\lambda_{ij}>0$ without changing $H_{ij}$:
$$2(\lambda_{ij} s_j- \lambda_{ij}s_i)^Tx = \lambda_{ij} \gamma_j-\lambda_{ij}\gamma_i \Leftrightarrow \lambda_{ij}\cdot 2( s_j- s_i)^Tx = \lambda_{ij} \cdot (\gamma_j-\gamma_i) \Leftrightarrow  2(s_j-s_i)^Tx = \gamma_j-\gamma_i.  $$
Of course, the hyperplane specified by $a_{ij}^Tx \leq \gamma_{ij}$ in Eq.~\eqref{eq:polyhedral-def} is similarly invariant under joint scaling of the left-hand and right-hand sides. 

These tools allow us to construct a linear program to check whether the cells specified by the given input can be represented as the cells of a power diagram. Recall Eq.~\eqref{eq:polyhedral-def}, which states the cells of the input as $P_i = \{x\in\reals^d : a_{ij}^T x \leq \gamma_{ij} \; \forall j \in J_i\}$. We have to match these with the description of the cells as specified in Eq.~\eqref{eq:gammacells}, $\cell(s_i)= \{\, x\in\reals^d : (s_j-s_i)^Tx\leq \gamma_j-\gamma_i \; \forall j \neq i  \}$. By the above, we may use any joint scaling of $a_{ij}$ and $\gamma_{ij}$ by a factor $\lambda_{ij}$ in this match.

Consider the following linear program for power diagram detection (PDD)
$$
\begin{array}{lrclcl}
  & \lambda_{ij} \cdot a_{ij}  & = & s_i-s_j &  & \,\,\,\, \forall i \leq k, \forall j \in J_i\\
(PDD) \;\;\;\;& \lambda_{ij}\cdot \gamma_{ij} & = & \gamma_j- \gamma_i & &   \,\,\,\, \forall i \leq k, \forall j \in J_i\\
& \lambda_{ij} & \geq & 1 & &  \,\,\,\, \forall i \leq k, \forall j \in J_i
\end{array}
$$
First, note that the $a_{ij}$ and $\gamma_{ij}$ are given. The $\{s_i\}_{i=1}^k \subseteq \reals^d$, the $\gamma_1,\dots,\gamma_k \in \reals$, and the $\lambda_{ij}$ $\,\,\,\,\forall i \leq k, \forall j \in J_i$ are variables. Further, note that we do not specify an objective function. We are only interested in finding a feasible solution, so a dummy objective function or a Phase-$0$ formulation of this program will be sufficient.

For all hyperplanes in the form $a_{ij}^Tx \leq \gamma_{ij}$ specified by the original input, this linear program checks whether $a_{ij}$ and $\gamma_{ij}$ can be jointly scaled to match Definition~\ref{def:pd2} of a power diagram. These are the first two lines of the program. 

Due to the invariance of power diagrams under joint scaling of all their parameters, the constraints $\lambda_{ij} \geq  1$ $\,\,\,\,\forall i \leq k, \forall j \in J_i$ may be imposed without losing the ability to construct a power diagram (in the cases where this is possible). These additional constraints guarantee that no trivial scaling $\lambda_{ij}=0$ and $s_j=s_i$ is feasible, as well as that no negative scaling $\lambda_{ij}<0$ is feasible, so that the direction of the separation of cells $P_i$ and $P_j$ cannot change.

We now use (PDD) to show the weakly polynomial solvability of Problem~\ref{problem}.

\begin{theorem}
  \label{thm:detect}
  For all rational input, Problem~\ref{problem} can be solved in weakly polynomial time by finding a feasible solution to (PDD). For a yes-instance, a representation of a power diagram is returned.
\end{theorem}

\begin{proof}
Recall that linear programs can be solved in weakly polynomial time through some variants of interior point methods, or the Ellipsoid method. As all the parameters in (PDD) are from the original input, it suffices to show that the number of variables and constraints is strongly polynomial in the size of the input. 

Let $p=\sum_{i=1}^k |J_i|$.  There are exactly $3p$ constraints and $k$ variable vectors of type $s_i\in \reals^d$, $k$ variables of type $\gamma_i$, and $p$ variables of type $\lambda_{ij}$, giving a total of $kd + k + p$ variables. 

The variables $\{s_i\}_{i=1}^k \subseteq \reals^d$ and $\gamma_1,\dots,\gamma_k \in \reals$ of a feasible solution represent a power diagram for a yes-instance, as in Definition~\ref{def:pd2}. Further, appropriate values for weights $w_i$ (as in Defintion~\ref{def:pd}) can be derived from the $\gamma_i$ through the equality $\gamma_i=\frac{1}{2}(s_i^Ts_i-w_i)$ for $i \leq k$. \qed
\end{proof}

In the next section, we exhibit a collection of applications in which power diagram detection arises. The practical problems essentially fit with the statement  of Problem~\ref{problem}, with the exception that often the domain is restricted, in the sense that the cells $P_i$ given are actually $P_i\cap D$ for some (full-dimensional) domain $D$. We therefore must address the problem of power diagram detection in restricted domains, where one asks whether the given cells can be expressed as a power diagram \emph{restricted to $D$}. We formally capture this problem in the following.

\begin{problem}[Power Diagram Detection for Restricted Domains]\label{problem-restricted}
  Let $D \subseteq \reals^d$ be a convex domain, and let $P=(P_1,\dots,P_k)$, $P_i\subseteq D$, be a polyhedral cell complex restricted to $D$ defined by a set of separating hyperplanes between the cells.
  Let the index sets $J^D_i = \{j : \dim (P_i \cap P_j) = d-1\}$ give the adjacency structure in the complex.
  Finally, let $a_{ij}\in\reals^d$ and $\gamma_{ij}\in \reals$ with
  \begin{equation}
    \label{eq:polyhedral-def-restricted}
    P_i = \{x\in D : a_{ij}^T x \leq \gamma_{ij} \; \forall j \in J_i^D\}.
  \end{equation}
\emph{Decide:} Is there a set of sites $\{s_i\}_{i=1}^k \subseteq \reals^d$ and $\{w_i\}_{i=1}^k \subseteq \reals$ such that $P_i = \cell(s_i) \cap D$ for all $i\in\{1,\ldots,k\}$?
\end{problem}

This problem variant can be resolved through a small tweak to (PDD): the index sets $J_i$ are replaced with the index sets $J_i^D$. We obtain the following linear program for restricted power diagram detection (r-PDD) for the detection of a power diagram in a restricted domain:
$$
\begin{array}{lrclcl}
  & \lambda_{ij} \cdot a_{ij}  & = & s_i-s_j &  & \,\,\,\, \forall i \leq k, \forall j \in J^D_i\\
(r{-}PDD) \;\;\;\;& \lambda_{ij}\cdot \gamma_{ij} & = & \gamma_j- \gamma_i & &   \,\,\,\, \forall i \leq k, \forall j \in J^D_i\\
& \lambda_{ij} & \geq & 1 & &  \,\,\,\, \forall i \leq k, \forall j \in J^D_i
\end{array}
$$

We now show that Problem~\ref{problem-restricted} is also weakly polynomially solvable. On one hand, we do not find this result surprising, as the structure of the associated linear program essentially remains the same as for Problem~\ref{problem}. However, one has to take special care to understand the impact of the restricted domain, and why the program does not give false-positive answers.

\begin{theorem}
  \label{thm:detect-restricted}
  For all rational input, Problem~\ref{problem-restricted} can be solved in weakly polynomial time by finding a feasible solution to (r-PDD).  For a yes-instance, a representation of a power diagram is returned.
\end{theorem}
\begin{proof}
The running time claim, and the return of a representation of a power diagram for yes-instances, follows analgously to the proof of Theorem~\ref{thm:detect}. The single, yet important difference lies in the restricted domain $D$. We have to consider how the index set $J^D_i$ in Problem~\ref{problem-restricted} relates to the cells $P_i$ in the restricted problem. More precisely, we have to make sure that a yes-answer to the problem corresponds precisely to those inputs, where the polyhedral cells in the domain $D$ come from the intersection of a power diagram in $\mathbb{R}^d$ with the domain $D$.

To this end, let $P'$ be a cell complex in $\mathbb{R}^d$ as defined in Problem~\ref{problem}, and let $J_i$ be the corresponding index sets. Further, let $J^D_i$ be the index sets as defined in Problem~\ref{problem-restricted}. Note $J^D_i\subseteq J_i$. Thus any feasible solution $(s,\gamma)$ to (PDD) (for index sets $J_i$) is also a feasible solution to (r-PDD) (for index sets $J^D_i$). Recall that, by Definition~\ref{def:pd2}, a feasible solution  $(s,\gamma)$ for (PDD) provides the parameters to represent a power diagram $P'$   in $\mathbb{R}^d$. This implies that if there exists a power diagram $P'$ in $\mathbb{R}^d$ with $P=P'\cap D$, the corresponding $(s,\gamma)$ is a feasible solution to both (PDD) and (r-PDD) and we correctly identify a yes-instance for Problem~\ref{problem-restricted}.

It remains that show that all no-instances are identified correctly, as well. Assume there is a feasible solution $(s,\gamma)$ to (r-PDD), and let $P'$ be the corresponding power diagram in $\mathbb{R}^d$. Let index sets $J_i$ and $J^D_i$ represent the adjacency of cells of $P'$ in $\mathbb{R}^d$, respectively in $D$, as defined in Problems~\ref{problem} and~\ref{problem-restricted}. We have to show that $P=P'\cap D$.

By construction, $P_i = \{x\in D : (s_j-s_i)^T x \leq \gamma_j-\gamma_i \; \forall j \in J_i^D\}$ and $P_i'\cap D=\{x\in D : (s_j-s_i)^T x \leq \gamma_j-\gamma_i \; \forall j \in J_i\}$ for all $i=1,\dots,k$. This gives $(P_i'\cap D)\subseteq P_i$ for all $i=1,\dots,k$. 

Now, recall that the cells $P_i$ decompose the convex domain $D$, i.e. $\bigcup_{i=1}^k P_i =D$ with $\text{int}(P_i)\cap\text{int}(P_j)=\emptyset$ for $i\neq j$. But as $P'$ is a power diagram of $\mathbb{R}^d$,  the $(P_i'\cap D)$ also decompose $D$, i.e. we have $\bigcup_{i=1}^k (P_i'\cap D) =D$ and $\text{int}(P'_i)\cap\text{int}(P'_j)=\emptyset$  for $i\neq j$. Together with $(P_i'\cap D)\subseteq P_i$, we obtain $P_i=(P_i'\cap D)$ for all $i=1,\dots,k$. This proves the claim.
\qed
\end{proof}

\section{Applications to Information Elicitation}\label{sec:applications}

In the domain of information elicitation, one wishes to design contracts or mechanisms to incentivize a self-minded agent to reveal their private information truthfully.
Often this private information comes in one of two varieties: a \emph{belief} about some future event, or a \emph{utility} or valuation about a particular commodity or outcome.
Power diagrams play a central role in the design of such mechanisms, as characterization theorems show that contracts or mechanisms are truthful (have the correct incentives) if and only if they partition the information space (or ``type space'') into cells that form a power diagram.
As the information space is typically not all of $\reals^d$, we will work with the restricted version, Problem~\ref{problem-restricted}.
In what follows, we show how to apply Theorem~\ref{thm:detect-restricted} to three such information elicitation scenarios where power diagrams arise.

\subsubsection*{Property Elicitation.}

One of the most basic information elicitation tasks is to design a contract to elicit some function, or \emph{property}, of an agent's belief.
The literature dates back to Savage~\cite{savage1971elicitation} and Osband~\cite{osband1985providing}, with its modern incarnation beginning with Lambert et al.~\cite{lambert2008eliciting,lambert2009eliciting}.
Concretely, consider a finite set of outcomes $\O$, the set of probability distributions $\Delta(\O)$, a finite set of possible \emph{reports} $\R$, and a property $\Gamma:\Delta(\O)\to 2^\R$ which designates a set of reports considered ``correct'' or ``desired'' for an agent with a particular belief.
For example, the map $\Gamma(p) = \argmax_{\o\in\O} p(\o)$ is the mode functional, which one notes is set-valued whenever there are multiple outcomes with the highest probability.

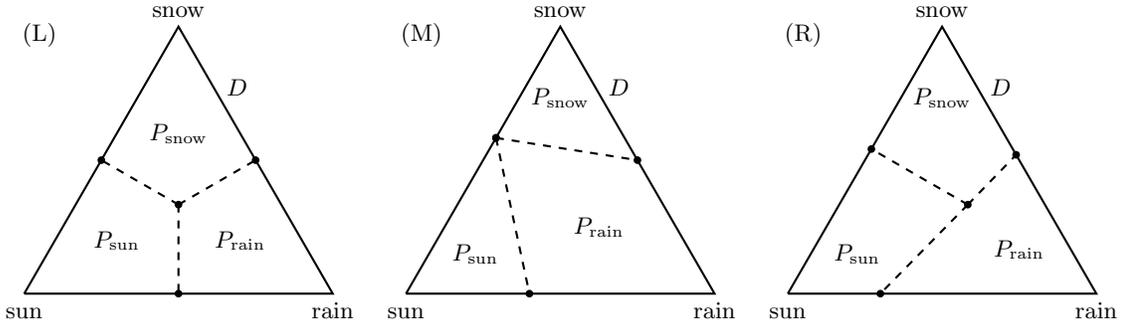
\begin{figure}[t]
  \begin{tikzpicture} [scale=\tikzfigscale, thick, tdplot_main_coords]
    \coordinate (orig) at (0,0,0);

    \pgfmathsetmacro{\ph}{4/12}
    \pgfmathsetmacro{\pm}{4/12}
    \pgfmathsetmacro{\pl}{4/12}
    \coordinate (p) at (\ph,\pm,\pl);
    \coordinate[label=below:sun\vphantom{i}] (h) at (1,0,0);
    \coordinate[label=below:rain] (m) at (0,1,0);
    \coordinate[label=above:snow\vphantom{i}] (l) at (0,0,1);

    \coordinate (hm) at (\ph+\pl/2,\pm+\pl/2,0);
    \coordinate (hl) at (\ph+\pm/2,0,\pl+\pm/2);
    \coordinate (ml) at (0,\pm+\ph/2,\pl+\ph/2);

    \draw[simplex] (h) -- (m) -- (l) -- (h);
    \draw[cell] (p) -- (hm);
    \draw[cell] (p) -- (hl);
    \draw[cell] (p) -- (ml);
    \fill (hm) circle [radius=0.5pt];
    \fill (hl) circle [radius=0.5pt];
    \fill (ml) circle [radius=0.5pt];
    \fill (p) circle [radius=0.5pt];

    \node at ($0.38*(m)+0.85*(l)$) {$D$};

    \node at ($0.4*(h)$) {$P_\text{sun}$};
    \node at ($0.4*(m)$) {$P_\text{rain}$};
    \node at ($0.4*(l)$) {$P_\text{snow}$};

    \placefiglabel{(L)}
  \end{tikzpicture}
  ~
  \begin{tikzpicture} [scale=\tikzfigscale, thick, tdplot_main_coords]
    \coordinate (orig) at (0,0,0);

    \pgfmathsetmacro{\ph}{5/12}
    \pgfmathsetmacro{\pl}{7/12}
    \pgfmathsetmacro{\pm}{0}
    \coordinate (p) at (\ph,\pm,\pl);
    \coordinate[label=below:sun\vphantom{i}] (h) at (1,0,0);
    \coordinate[label=below:rain] (m) at (0,1,0);
    \coordinate[label=above:snow\vphantom{i}] (l) at (0,0,1);

    \coordinate (hm) at (0.6,0.4,0);
    \coordinate (hl) at (\ph+\pm/2,0,\pl+\pm/2);
    \coordinate (ml) at (0,0.5,0.5);

    \draw[simplex] (h) -- (m) -- (l) -- (h);
    \draw[cell] (p) -- (hm);
    \draw[cell] (p) -- (hl);
    \draw[cell] (p) -- (ml);
    \fill (hm) circle [radius=0.5pt];
    \fill (hl) circle [radius=0.5pt];
    \fill (ml) circle [radius=0.5pt];
    \fill (p) circle [radius=0.5pt];

    \node at ($0.38*(m)+0.85*(l)$) {$D$};

    \node at ($0.55*(h)$) {$P_\text{sun}$};
    \node at ($0.25*(m)$) {$P_\text{rain}$};
    \node at ($0.6*(l)$) {$P_\text{snow}$};

    \placefiglabel{(M)}
  \end{tikzpicture}
  ~
  \begin{tikzpicture} [scale=\tikzfigscale, thick, tdplot_main_coords]
    \coordinate (orig) at (0,0,0);

    \pgfmathsetmacro{\ph}{3/12}
    \pgfmathsetmacro{\pl}{4/12}
    \pgfmathsetmacro{\pm}{5/12}
    \coordinate (p) at (\ph,\pm,\pl);
    \coordinate[label=below:sun\vphantom{i}] (h) at (1,0,0);
    \coordinate[label=below:rain] (m) at (0,1,0);
    \coordinate[label=above:snow\vphantom{i}] (l) at (0,0,1);

    \coordinate (hm) at (0.7,0.3,0);
    \coordinate (hl) at (\ph+\pm/2,0,\pl+\pm/2);
    \coordinate (ml) at (0,0.48,0.52);

    \draw[simplex] (h) -- (m) -- (l) -- (h);
    \draw[cell] (p) -- (hm);
    \draw[cell] (p) -- (hl);
    \draw[cell] (p) -- (ml);
    \fill (hm) circle [radius=0.5pt];
    \fill (hl) circle [radius=0.5pt];
    \fill (ml) circle [radius=0.5pt];
    \fill (p) circle [radius=0.5pt];

    \node at ($0.38*(m)+0.85*(l)$) {$D$};

    \node at ($0.55*(h)$) {$P_\text{sun}$};
    \node at ($0.5*(m)$) {$P_\text{rain}$};
    \node at ($0.6*(l)$) {$P_\text{snow}$};

    \placefiglabel{(R)}
  \end{tikzpicture}
  \caption{\textbf{(L)} The cell decomposition corresponding to the mode functional, which is a power diagram (indeed, a Voronoi diagram) intersected with $D$, the probability simplex over outcomes $\{$sun, rain, snow$\}$ projected onto $\reals^2$.  (Take e.g. sites at the corners of the simplex, with equal weights.)
    \textbf{(M)} Another cell decomposition which is a power diagram intersected with $D$, and therefore corresponds to an elicitable property.  Note that the index set would be $J_{\text{sun}} = \{$rain, snow$\}$ in Problem~\ref{problem}, whereas $J_{\text{sun}}^D = \{$rain$\}$ in Problem~\ref{problem-restricted}.
    \textbf{(R)} A cell decomposition which is not a power diagram intersected with $D$, and therefore not an elicitable property.}
  \label{fig:example-simplex}
\end{figure}

To elicit the property $\Gamma$ from an agent, we wish to design a contract $S:\R\times\O\to\reals$ which determines the payment to the agent once the outcome materializes.
The protocol is thus: the score $S$ is announced, the agent reports some $r\in\R$, the outcome $\o\in\O$ is revealed, and the agent is paid $S(r,\o)$.
We say that the score $S$ \emph{elicits} the property $\Gamma$ if for all beliefs $p\in\Delta(\O)$, we have
\begin{equation}
  \label{eq:prop-elic}
  \Gamma(p) = \argmax_{r\in\R} \E_{\o\sim p}[S(r,\o)]~,
\end{equation}
that is, the agent maximizes their expected score according to their belief $p$ by reporting $r\in\Gamma(p)$.
For example, if one would like to know which of sun, rain, or snow, is most likely for the weather tomorrow, one would set $\R=\O=\{$sun, rain, snow$\}$ and offer to pay the agent $S(r,\o) = \ones\{r=\o\}$, for the simple reason that this score elicits the mode: $\argmax_r \E_{\o\sim p}[\ones\{r=\o\}] = \argmax_r p(r)$.

A natural question is thus the following: for which properties $\Gamma$ can one devise a score which elicits it?
Perhaps surprisingly, the answer is simple: the elicitable properties $\Gamma$ are precisely those for which the partition $P=(P_r)_{r\in\R}$ given by $P_r = \{p : r \in \Gamma(p)\}$ forms a power diagram~\cite{lambert2009eliciting,frongillo2014general}.
More precisely, after projecting to $\reals^{|\O|-1}$, e.g. by dropping the last coordinate, the partition forms a power diagram intersected with the (projected) probability simplex.
See Figure~\ref{fig:example-simplex}(L) for an illustration of this projection for the mode.
The rough intuition for this result is as follows: Eq.~\eqref{eq:cell} can be expressed via an $\argmin$ over $k$ affine functions of $x$ by dropping the irrelevant $\|x\|^2$ term, and then as the negation of an $\argmax$ of affine functions.
As the expected score is linear (and therefore affine) in the agent's belief $p$, the distributions $p$ which share the same optimal report (i.e.,\ the same $\argmax$) must form a cell of a power diagram.

With this characterization at hand, as a practical matter, it becomes relevant to test whether a given polyhedral cell complex, dividing the probability simplex into regions with the same report, forms a power diagram.
This test can be done via Theorem~\ref{thm:detect-restricted}, where the restricted domain $D$ is the probability simplex, $D = \{x\in\reals^d : \sum_{i=1}^d x_i = 1, x_i \geq 0 \;\;\; \forall i\leq d \}$.
Note that the restricted domain can indeed change the index sets from (PDD) to (r-PDD), and even when $P_i \cap P_j \neq \emptyset$, one could still have $j \notin J^D_i$, as we illustrate in Figure~\ref{fig:example-simplex}(M).

Essentially the same formalism as described above arises in machine learning as well, in the context of designing loss functions for classification or ranking tasks~\cite{bartlett2006convexity,ramaswamy2013convex,agarwal2015consistent}.
Here one may wish to assign certain distributions over the labels (or classes) to different reports, and one may ask, for which such assignments does there exist a loss function which, when minimized, would yield these reports.
This is again the same question as whether the property corresponding to this assignment is elicitable, where we simply negate the score to obtain a loss.

\subsubsection*{Peer Prediction.}

A problem closely related to property elicitation is peer prediction, where one has access to several agents rather than one, but no direct access to the ground truth.
Instead of designing a contract which scores an agent's report based on some observed outcome, one will never see the outcome, and therefore must score the agents based on each other's reports~\cite{miller2005eliciting,witkowski2012robust}.

To make headway, one typically assumes some structure about the underlying Bayesian (or pseudo-Bayesian) process through which agents form their beliefs.
In particular, one typically assumes that agents receive some \emph{signal} $S\in\O$ about the true outcome (like the quality of a hotel, or the correct label for an image classification task), and from this signal form a posterior belief $p(\cdot|S)$ about the true outcome, or about what another agent's signal was.
By making assumptions about the possible values of this posterior, one can design mechanisms which have a truthful equilibrium, in which each agent simply reports their signal $S$.

Constraints on the possible posteriors in the literature often take the form $p(\cdot|S=\o)\in P_\o$ for some sets $\{P_\o : \o\in\O\}$, dubbed a \emph{belief model constraint} in Frongillo and Witkowski~\cite{frongillo2017geometric}.
Their work shows that there exists a mechanism with a truthful equilibrium if and only if the sets $\{P_\o\}$ form a power diagram~\cite[Corollary 3.5]{frongillo2017geometric}.
(More precisely, there must exist a power diagram with sites $s_\o$ such that $P_\o\subseteq\cell(s_\o)$ for all $\o\in\O$, but we restrict attention to \emph{maximal} constraints, where every distribution is allowed to be the posterior following some signal.)
The problem of detecting whether a belief model constraint supports a mechanism, and constructing said mechanism if so, thus reduces to detecting and constructing a power diagram in a restricted domain (again the probability simplex), to which Theorem~\ref{thm:detect-restricted} applies.

\begin{figure}
  \centering
  \begin{tikzpicture}[thick,scale=0.8]
    \path[->] (0,0) edge node[left]{$t(\o_2)$} (0,4) edge node[below]{$t(\o_1)$} (4,0);
    \draw[dashed] (2,0) -- (2,4) (2,2) -- (4,2);
    \node[anchor=center] at (1,2) {$p_1$};
    \node[anchor=center] at (3,1) {$p_3$};
    \node[anchor=center] at (3,3) {$p_2$};
  \end{tikzpicture}
  \qquad
  \begin{tikzpicture}[thick,scale=0.8]
    \path[->] (0,0) edge node[left]{$t(\o_2)$} (0,4) edge node[below]{$t(\o_1)$} (4,0);
    \draw[dashed] (2,0) -- (2,4) (0,2) -- (4,2);
    \node[anchor=center] at (1,1) {$p_1$};
    \node[anchor=center] at (1,3) {$p_0$};
    \node[anchor=center] at (3,1) {$p_3$};
    \node[anchor=center] at (3,3) {$p_2$};
  \end{tikzpicture}
  \caption{An allocation rule ``skeleton'' which cannot be completed to an implementable allocation rule for any distinct choices of $p_i$ (left), and one which could be so completed for appropriate choices of $p_i$ (right).  Here $t$ denotes the type of an agent, and $t$ being the cell corresponding to $p_i$ implies the allocation $f(t) = p_i$.}
  \label{fig:example-mechanism}
\end{figure}
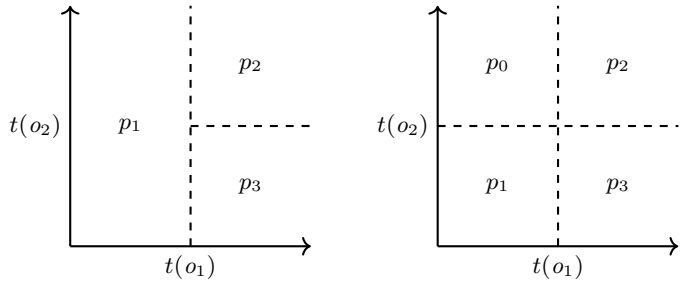

\subsubsection*{Mechanism Design.}

In mechanism design, one wishes to design an algorithm to choose an outcome based on the reports of the participants, but in a way which is robust to strategic misreports.
From the well-known \emph{revelation principle}, we may assume without loss of generality that the reports or ``bids'' submitted to a mechanism take the form of utility functions: each participant reports their utility for each possible outcome.
From the reported utilities (called \emph{types}) of the agents, the mechanism then chooses a distribution over outcomes (called the \emph{allocation}) as well as a payment owed the mechanism by each participant.
For simplicity, we will consider the case of a single agent.

Given the agent's type $t\in\reals^\O$ which encodes their utility $t(\o)$ following each outcome $\o$, the mechanism wishes to choose a random allocation from some set $\O$ of outcomes, as well as the amount the agent should pay.
Formally, given a finite outcome space $\O$ and a convex \emph{type space} $\T \subseteq \reals^\O$, a (direct, randomized) \emph{mechanism} is a pair $(f,p)$ where $f:\T\to\Delta_\O$ is the \emph{allocation rule} and $p:\T\to\reals$ is the \emph{payment} function.
We typically assume that the agent's utility is \emph{quasi-linear} in the sense that their net utility upon allocation $\o$ and payment $c$ is the difference $t(\o) - c$.
Note that the expected utility can be written $U(t',t) = f(t')\cdot t - p(t')$, where the inner product is between elements of $\reals^\O$, where we consider $f(t') \in \Delta_\O \subset \reals^\O$ to be the vector form of the allocation distribution.
We say the mechanism $(f,p)$ is \emph{truthful} if $U(t',t) \leq U(t,t)$ for all $t,t' \in \T$.

A fundamental question in mechanism design is \emph{implementability}: given a desired allocation rule $f$, can one find a payment function $p$ such that the pair $(f,p)$ is a truthful mechanism?
In other words, can one design payments to implement the desired allocation rule while being robust to the incentives of the agents?
(If given the allocations themselves, one can use the fact that weak monotonicity (WMON) is sufficient and perform checks in $O(n^2)$ time~\cite{saks2005weak}.)
We consider the following variant of this classic question: given the ``skeleton'' of an allocation rule, that decides upon ``cells'' of types to assign the same allocation but not what that allocation should be, could there be any assignment of allocations to cells such that the resulting allocation rule is implementable?
See Figure~\ref{fig:example-mechanism} for an illustration.
More formally, given the skeleton $g:\T\to\{1,\ldots,m\}$, could there be any implementable allocation rule of the form $f = h \circ g$ for some $h:\{1,\ldots,m\}\to\Delta(\O)$?
From~\cite{saks2005weak,frongillo2014general}, such skeletons must form a power diagram, and thus, given $g$ in the form~\eqref{eq:polyhedral-def-restricted} where $D=\T$ is the type space, this problem can also be solved as stated in Theorem~\ref{thm:detect-restricted}.

\section{Conclusion}
\label{sec:conclusion}

We have presented a simple linear program for power diagram detection, the problem of deciding whether a polyhedral cell decomposition given by linear equalities can be represented as a power diagram. The problem of power diagram detection arises in information elicitation applications, and we have detailed three such applications, to property elicitation, peer prediction, and mechanism design. In each of these problems, the domain is typically restricted to a subset $D$ of $\reals^d$, and thus we have also addressed the restricted power diagram detection problem (Problem~\ref{problem-restricted} and Theorem~\ref{thm:detect-restricted}), to determine whether the given complex can be represented as a power diagram restricted to $D$. Combined, our results give a (weakly) polynomial time algorithm to determine whether a given property is elicitable, whether a given belief model constraint corresponds to a truthful peer prediction mechanism, or whether a given allocation ``skeleton'' could extended to an implementable allocation rule.

As a practical concern, we note that information elicitation problems do not necessarily exhibit polyhedral representations of their cells. More precisely, it is possible that the cells of a proposed property, belief model constraint, or allocation rule, are not convex polyhedra, and thus could not possibly be represented as those of a power diagram. Even when they allow a polyhedral representation, the input may be given in a form such that this is not obvious. For example, given a representation of a polyhedron as the convex hull of its vertices, it typically is not efficient to devise a representation as the intersection of halfspaces. In many applications, expert knowledge will allow a viable resolution of these issues, before the methods presented in this paper will be applied. In general however, they pose significant challenges before obtaining a broader detection algorithm in practice.


\bibliographystyle{plain}
\bibliography{diss1}

\end{document}